\documentclass[12pt,a4paper]{article}
\pdfoutput=1
\usepackage{hyperref}

\usepackage{amsfonts}
\usepackage{amsmath}

\usepackage{bm}
\usepackage{mathdots}
\usepackage{graphicx}
\usepackage{amsthm}

\usepackage{xcolor} % use color to make it easy to see revisions

\newcommand{\ud}{\mbox{d}}
\newtheorem{theorem}{Theorem}[section]
\newtheorem{corollary}{Corollary}[theorem]
\newtheorem{lemma}{Lemma}[section]
\usepackage{cleveref}

\usepackage{authblk}

\title{The reflectionless properties of Toeplitz waves and Hankel waves: an analysis via Bessel functions }

\author[1,2]{Kevin Burrage}
\author[1]{Pamela Burrage}
\author[3]{Shev MacNamara}

\affil[1]{ARC Centre of Excellence for Mathematical and Statistical Frontiers, and Mathematical Sciences School,\;\;\;\;\;\;\;\;\;\;\;\;\;\;\;\;\;
      Queensland University of Technology (QUT), Australia.}
\affil[2]{Visiting Professor, Computer Science Department,  \;\;\;\;\;\;\;\;\;\;\;\;\;\;\;\;\; University of Oxford, UK.}
\affil[3]{ARC Centre of Excellence for Mathematical and Statistical Frontiers, School of Mathematical and Physical Sciences, University of Technology Sydney, Australia}

\begin{document}

\maketitle

\begin{abstract}
We study reflectionless properties at the boundary  for the wave equation  in one space dimension and time, in terms of a  well-known matrix that arises from a simple discretisation of space. 
It is known that all matrix functions of the familiar second difference matrix representing the Laplacian in this setting are the sum of a Toeplitz matrix and a Hankel matrix.
The solution to the wave equation is one such matrix function. 
Here, we study the behaviour of the corresponding waves that we call Toeplitz waves and Hankel waves. 
We show that these waves can be written as certain linear combinations of even Bessel functions of the first kind.
We find exact and explicit formulae for these waves. 
We also show that the Toeplitz and Hankel waves are reflectionless on even, respectively odd, traversals of the domain.
Our analysis naturally suggests a new method of computer simulation that allows control, so that it is possible to choose --- in advance --- the number of reflections.
An attractive result that comes out of our analysis is the appearance of  the well-known shift matrix, and also other matrices that might be thought of as Hankel versions of the shift matrix.
By revealing the algebraic structure of the solution in terms of shift matrices, we make it clear how the Toeplitz and Hankel waves are indeed reflectionless at the boundary on even or odd traversals.
Although the subject of the reflectionless boundary condition has a long history, we believe the point of view that we adopt here in terms of matrix functions is new.  
\end{abstract}

~\\

\textbf{keywords}  \;\;\;\;\;\;\;\;\;\;\;
  Toeplitz waves, \;\; Hankel Waves, \;\; One--way waves, \;\; Bessel functions,  \;\; Matrix functions  ~\\


~\\
~\\

\section{Introduction}
``One--way waves'' are similar to solutions of wave equations but they move in only one direction, and exhibit no propagation in the opposite direction.
They are important in the subject of absorbing boundary conditions  \cite{EngquistMajda77}.
Trefethen \& Halpern  study the well-posedness of one--way wave equations, and they give a nice  overview of the history  \cite{TrefethenHalpern1986}. 
The literature on reflectionless boundary conditions, and on applications of Bessel function expansions to waves, is extensive. 
We do not attempt a survey here. 
An incomplete list of authors includes: Grote, Keller, Givoli, Bayliss \& Turkel, Higdon, Ting \& Miksis, Lubich \& Schadle, Alpert, Greengard \& Hagstrom. 
Particular attention is paid to the issue of whether or not the proposed reflectionless boundary conditions are exact or approximate, and whether or not the conditions are local or global.
%See, e.g., Grote \& Keller \cite{GroteKeller1995}. 
There are also important connections to the Kirchoff formula and to the Dirichlet--to--Neumann map, and to Born series, that we do not consider here.

We distinguish three settings:
\begin{itemize}
\item Fully continuous: both time and space are continuous.
 The reflectionless boundary condition is the Sommerfeld condition, which in our setting might be applied by imposing  
\begin{equation}
\frac{\ud u}{\ud x } = \pm \frac{\ud u}{\ud t }
\label{eq:orr-sommerfeld}
\end{equation}
on the solution $u(x,t)$ at the left and right boundaries.
The choice of the plus or minus sign depends on whether the wave is traveling left or right. 
 Notice this condition \eqref{eq:orr-sommerfeld} is local in time and local in space. 
\item Fully discrete: Both time and space have been discretized. 
Engquist  \&  Majda \cite[Equation (5.3)]{EnquistMajda79} address the issue of the nature of the condition in this setting. 
Their reflectionless boundary condition is not a simple expression. 
It is not local in time and it is not local in space.
 
\item Semi-discrete: time remains continuous and space is discretized.
Intuitively, we expect the conditions to be local in time, and non-local in space.
This situation has been studied by Halpern   \cite[Section 3]{halpern1982absorbing}. 
%\item A set of masses connected by  linear springs, which can be generalised to higher dimensions through an array of springs. 
%The resulting framework is then a linear system of ordinary differential equations.
\end{itemize}
%\colorbox{yellow}{\parbox{0.98\textwidth}{
A well-known model in these settings is a set of masses connected by  linear springs, which can be generalised to higher dimensions through an array of springs. 
The resulting framework is then a linear system of ordinary differential equations.
%}}

\noindent
The focus of our article is on the semi-discrete setting, where time remains continuous.
 However, from equation \eqref{eq:eq6} onwards we find it more convenient to evaluate our expressions at certain discrete time points.
 (Note that our evaluations here are \textit{exact} for the semi-discrete setting, as opposed to what is commonly done in the literature in the  `fully-discrete' setting where some further approximation is introduced.)
In our semi-discrete setting, the solution to the wave equation can be thought of as a matrix function \cite{StrMac14}.
Very recently Nadukandi \& Higham have shown how to efficiently compute these wave--kernel matrix functions \cite{WaveKernelsHigham}.

It is important to note that our purpose in this article is not to study the reflectionless boundary condition \textit{per se}, nor to study the big field of inverse problems to which it relates; these topics already have an extensive literature.
The boundary condition in particular,  is an old problem that has been studied by many authors; among the early literature, perhaps the work of  Halpern   \cite{halpern1982absorbing} is closest to our own framework. 
Instead, our purpose is to reveal the properties of the Toeplitz waves and the Hankel waves that we will introduce later. 
During the course of our analysis, it transpires that these Toeplitz and Hankel waves possess some attractive reflectionless properties.
%Our analysis is a lucid explanation of those properties and makes them particularly explicit in terms of algebraic structures and exact formulas that we find.

We revisit the problem with a new perspective, by carefully revealing the algebraic structure of those matrix functions.
We find particularly explicit and exact formulas.
An attractive  outcome of our perspective is to show that the travelling wave can be considered as the sum of two waves that we call the Toeplitz wave and the Hankel wave. 
We show that these waves can be written as certain linear combinations of even index Bessel functions of the first kind. These expansions, as we will explain, make it clear that the Toeplitz wave and Hankel wave are reflectionless on even, respectively odd, traversals of the domain.

In \cref{sec:2}, we give some background on Toeplitz and Hankel matrix functions.
 In \cref{sec:3}, we show how to write a travelling wave  as the sum of two waves that we call a Toeplitz wave and a Hankel wave. 
In \cref{sec:4}, we give analytic expressions for the Toeplitz and Hankel waves in terms of certain sums of even index Bessel functions of the first kind.
In \cref{sec:5}, we show  simulations of these waves using the methods of Nadukandi \& Higham \cite{WaveKernelsHigham}, and also using methods inspired by our analytic expressions.
With our analysis, we are now able to algebraically explain interesting behaviours that have previously been observed in numerical simulations in the literature  \cite{macStrang2016}.
We also showcase another benefit of our analysis, namely that it suggests a method of simulation that allows us to choose, in advance, the number of reflections.
  The paper concludes in \cref{sec:discussion} with some observations and discussion for future work.

\section{Toeplitz--plus--Hankel matrix functions}
\label{sec:2}
This section collects together some of the key results from \cite{StrMac14} concerning Toeplitz-plus-Hankel matrix functions that we will need later.

The central difference approximation to the Laplacian, which in one space dimension is simply the second derivative, $-\frac{\partial^2}{\partial x^2}$, is $\bm{K}/h^2$, where $\bm{K}$ is the $N \times N$ tridiagonal, symmetric positive definite Toeplitz matrix:
\begin{equation}
\bm{K} =
\left(
\begin{tabular}{ccccc}
$2$ & -1 \\
 -1 & $2$ & -1 \\
  & $\ddots$ & $\ddots$ & $\ddots$ \\
 & & -1 & $2$ & -1 \\
  &  & &   -1 & 2
\end{tabular}
\right)
\label{eq:K}
\end{equation}
and 
\begin{equation}
h = \frac{1}{N+1}.
\label{eq:h}
\end{equation}
We can diagonalize the matrix $\bm{K}  =  \bm{V} \bm{\Lambda} \bm{V}^\top  =  \sum_{k=1}^N \lambda_k \bm{v}_k \bm{v}_k^\top $ where the eigenvalues appear in the diagonal matrix $\bm{\Lambda}$ and where the eigenvectors form the columns of $\bm{V}$.
Given a function $f$ of a scalar argument, we define a matrix function \cite{NicholasHighamBook08}  via  diagonalization:
\begin{equation}
f(\bm{K}) = \bm{V} f(\bm{\Lambda}) \bm{V}^\top = \sum_{k=1}^N f(\lambda_k) \bm{v}_k \bm{v}_k^\top.
\label{eq:matrix:function:diagonalization}
\end{equation}
The eigenvalues of the matrix in \eqref{eq:K} are:
\begin{equation}
 \lambda_k = 2 - 2 \cos (k \pi h),  \qquad k = 1, \dots, N, 
\label{eq:eigenvalues}
\end{equation}
and the eigenvectors are:
\begin{equation}
\bm{v}_k = \sqrt{2h} \Big( \sin (k \pi h) , \sin (2 k \pi h), \dots , \sin (N k \pi h)   \Big)^\top, \qquad k = 1, \dots, N.
\label{eq:eigenvectors}
\end{equation}
With these explicit expressions, the $(m,n)$ entry of the matrix function is 
\begin{equation}
f(\bm{K})_{m,n} = 2h  \sum_{k=1}^{N} f(\lambda_k) \sin(m k \pi h) \sin(n k \pi h).
\label{eq:matrix:function:K}
\end{equation}

 \noindent %\colorbox{yellow}{\parbox{0.99\textwidth}{ 
 Recall that Toeplitz matrices are those with constant diagonals, whereas Hankel matrices have constant anti-diagonals. 
We now define the strong Toeplitz-plus-Hankel property (as in \cite{StrMac14}): a matrix has the  strong Toeplitz-plus-Hankel property when all matrix functions are  the sum of a Toeplitz matrix and a Hankel matrix.
For the purpose of this definition and this article, we only consider matrices that are diagonalizable, and by all matrix functions, we only consider those functions that come via diagonalization, as defined above.
 The  strong Toeplitz-plus-Hankel property is equivalent to requiring that each rank one projection matrix coming from the eigenvectors, can be written as a sum of a Toeplitz matrix and a Hankel matrix.
 The matrix $\bm{K}$ has this strong property, as we will now demonstrate.
 %}}

Entries of the  rank one symmetric matrices, denoted $\bm{v}_k \bm{v}_k^\top$ that project onto eigenspaces, are products of sines, of the form $\sin(m \theta) \sin(n \theta)$.
Recall the trigonometric identity
\begin{equation}
2 \sin \theta_1 \sin \theta_2 = \cos (\theta_1-\theta_2) - \cos(\theta_1+\theta_2).
\end{equation}
Thus each entry in $\bm{v}_k \bm{v}_k^\top$ can be written as the sum of a term  of the form $\cos((m-n)\theta )$ that leads to a Toeplitz matrix, and another term of the form $ \cos((m+n)\theta ) $ that leads to a Hankel matrix.
Thus, for $k=1,\ldots,N$, the rank one matrix 
\begin{equation}
\bm{v}_k \bm{v}_k^\top   = \bm{T}_k + \bm{H}_k  
\label{eq:V=T+H}
\end{equation}
is the sum of a Toeplitz matrix, with  $(m,n)$ entries
\begin{equation}
 \big( \bm{T}_k \big)_{m,n} = \; \, h \cos \big( (m\bm{-}n) k \pi h \big) 
 \label{eq:T:part}
\end{equation}
and a Hankel matrix
\begin{equation}
  \big( \bm{H}_k \big)_{m,n} = \bm{-} h \cos \big( (m\bm{+}n) k \pi h \big), 
   \label{eq:H:part}
\end{equation}
recalling that $h = \frac{1}{N+1}$.

A sum of Toeplitz matrices is again a Toeplitz matrix, and similarly Hankel matrices are closed under addition.
Recalling \eqref{eq:matrix:function:diagonalization}, \eqref{eq:matrix:function:K} and \eqref{eq:V=T+H}, we see that as $\bm{K}$ has the strong Toeplitz-plus-Hankel property,  all matrix functions of $\bm{K}$ are the  sum of a Toeplitz matrix and a Hankel matrix:
\begin{equation}
 f(\bm{K})    =  \sum_{k=1}^N f(\lambda_k) (\bm{T}_k + \bm{H}_k )  
   =    \underbrace{ \sum_{k=1}^N f(\lambda_k) \bm{T}_k }_\text{Toeplitz}  \;\;  + \;\; \underbrace{\sum_{k=1}^N f(\lambda_k)  \bm{H}_k }_\text{Hankel} .
 \label{eq:diagonalize:TplusH}
\end{equation} 

\noindent %\colorbox{yellow}{\parbox{0.98\textwidth}{
A few observations are noteworthy.
In general, a matrix cannot be written exactly as the sum of a Toeplitz matrix and a Hankel matrix.
%}}

\noindent  When it is possible to express a matrix exactly as the sum of  a Toeplitz matrix  and a Hankel matrix, the Toeplitz-plus-Hankel splitting is not unique.
The reason is that there is a two dimensional space of matrices that are simultaneously both Toeplitz and Hankel (examples of such matrices are displayed in \eqref{eq:A:B:simultaneous:TplusH}).
One possible basis for that space is formed by two matrices: the all ones matrix, in which every entry is 1, and the checkerboard matrix in which every entry is of magnitude 1, together with an alternating pattern of $\pm$ signs.
Thus there is a somewhat arbitrary choice as to where to include these matrices (they could go into either the Toeplitz part or the Hankel part, possibly also with some weighting).
In this article we always make the particular choice that is implicit when we use  \eqref{eq:T:part} and \eqref{eq:H:part}.

It is helpful to consider the action of the Toeplitz and Hankel parts on a vector `input.'
To illustrate this action, consider the so-called shift matrix, which is an example of a Toeplitz matrix.
In this particular example, we see the action is  a `forward shift' that preserves the orientation of the input:
\begin{equation}
\left(
\begin{tabular}{ccccc}
 &  \\
 1 &  &  \\
  &1  &  &  \\
  & & 1  &  \\
  & &   &1  &
\end{tabular}
\right) 
\left(
\begin{tabular}{c}
 0  \\
 a  \\
 b   \\
 c  \\
 0
\end{tabular}
\right)
=
\left(
\begin{tabular}{c}
 0 \\
 0 \\
 a   \\
  b   \\
  c   \\
\end{tabular}
\right).
\end{equation}
%\colorbox{yellow}{\parbox{0.98\textwidth}{
Define $N \times N$ shift matrices $E_1, \cdots, E_{N-1}$ where $E_k$ corresponds to the matrix with 1 on the $k$th upper subdiagonal and 0 elsewhere.
We will need these matrices later in \eqref{eq:eqEk:shift:matrix}.
Let  $E_0=I$ be the identity matrix.
These matrices and their transposes form a natural basis, $\{E_0, E_1, \ldots,  E_{N-1}, E_1^\top, \ldots  E_{N-1}^\top\}$, for the space of $N \times N$ Toeplitz matrices.
That is, an arbitrary  $N \times N$ Toeplitz matrix $T$, with first row $t_0, t_1, \ldots t_{N-1}$ and first column $t_0, t^1  \ldots, t^{N-1}$, is a linear combination of these shift matrices: $T=t_0 I  +  t_1 E_1 \ldots t_{n-1} E_{N-1} + t^1E_1^\top + \ldots t^{N-1} E_{N-1}^\top$.
% (where the weights in the combination come from the first row and the first column of the Toeplitz matrix).
Thus, we can think of the action of the Toeplitz matrix as a linear combination of these shifts \cite{gray2006toeplitz}.
Likewise, to illustrate the action of a Hankel matrix, consider the `Hankel version of the shift matrix.'
In this example, we see a `backward shift' (which is also known as a reversal permutation) that reverses the orientation of the input:
%}}
\begin{equation}
\left(
\begin{tabular}{ccccc}
& & &  1&  \\
& &   1&  &  \\
 &  1&  &  &  \\
1 & & \\
& 
\end{tabular}
\right)
\left(
\begin{tabular}{c}
 0  \\
 a  \\
 b   \\
 c  \\
 0
\end{tabular}
\right)
=
\left(
\begin{tabular}{c}
 c   \\
  b   \\
  a   \\
  0 \\
  0
\end{tabular}
\right).
\end{equation}
%\colorbox{yellow}{\parbox{0.98\textwidth}{
We define a set  of such $N \times N$ matrices, with ones on a backwards diagonal and which we denote by  $F_k$, and which are depicted  later in \eqref{eq:backward:hanel:shifts}.
These matrices are related to the shift matrices by multiplying by the anti-identity matrix $J$, that is $F_k = J E_k$ for $k=1,\ldots,N-1$, $F_N=J E_0 = J I =J$, and $F_{N+k} = J E_k^\top$ for $k=1,\ldots,N-1$.
The set of matrices $\{F_k \}$ form a natural basis for the space of $ N \times N$ Hankel matrices.
That is, an arbitrary Hankel matrix $H$ can be written as a linear combination $H = h_1F_1 + \ldots h_{2N-1}F_{2N-1}$, where the weights $h_k$ in the combination come from the backward diagonals of the Hankel matrix.
%}}

  \section{The wave equation is Toeplitz--plus--Hankel}
  \label{sec:3}
This section summarizes observations from \cite{macStrang2016} that we will need later.
It transpires that the Toeplitz part of the solution to the wave equation  can be thought of as a type of solution that does not feel the boundaries on even traversals of the domain, whereas the Hankel part of the solution does not feel the boundaries on odd traversals of the domain.

Consider the wave equation 
\begin{equation}
\frac{\partial^2 }{\partial t^2} \bm{u}  = \frac{\partial^2}{\partial x^2} \bm{u}
\label{eq:wave}
\end{equation}
on  the domain $-1 \leq x \leq 1$ with zero Dirichlet boundary conditions $u(-1,t) = u(1,t) =0$.
In the semi-discrete setting, the continuous wave equation \eqref{eq:wave} is approximated on an equally spaced grid of points $-1, \ldots, - \Delta x,  0, \Delta x, \ldots, 1$, with spacing $\Delta x = 2/(N-1)$, by the linear system of ordinary differential equations
\begin{equation}
\frac{\ud^2}{\ud t^2} \bm{u} = - \frac{\bm{K}}{\Delta x^2} \bm{u}
\label{eq:wave:semi:discrete}
\end{equation}
 with the matrix from \eqref{eq:K}.
 It can be quickly checked by differentiating twice that a solution to our semi-discrete model in \eqref{eq:wave:semi:discrete} is the matrix function
 \begin{equation}
\bm{u}(t) = f(\bm{K}) \bm{u}(0)  = \cos \left( \pm t \frac{\sqrt{\bm{K}}  }{ \Delta x} \right) \bm{u}(0). 
\label{eq:solution}
\end{equation}
Here we are also using the square root matrix function, and
 we restrict attention to the special case of a wave equation involving a symmetric graph Laplacian, of which the matrix $\bm{K}$ in \eqref{eq:K} is an example.
This solution \eqref{eq:solution} could be compared with the representations of the solutions presented by Nadukandi \& Higham \cite{WaveKernelsHigham}, who are able to efficiently compute solutions to the wave equation, even in the more general situation where  the graph Laplacian matrix in question is not symmetric.

A second order differential equation such as \eqref{eq:wave:semi:discrete} requires two initial conditions.
In our solution \eqref{eq:solution}, we are assuming an initial condition $\bm{u}(0)$, and we are also assuming zero initial velocity.
If the initial velocity, $\bm{u}^{'}(0)$, is not zero then the solution involves an extra term: 
\begin{equation}
 \bm{u}(t) = \cos \left(t \frac{ \sqrt{\bm{K}} }{  \Delta x}\right)  \bm{u}(0) +    \Delta x \> \bm{K}^{-\frac{1}{2}} \sin \left(t \frac{\sqrt{\bm{K}} }{ \Delta x}\right)  \;  \bm{u}^{'}(0) .
\label{eq:semidiscrete:wave:eqn:sol:dalembert}
\end{equation}
%Apart from boundary conditions, this solution in \eqref{eq:semidiscrete:wave:eqn:sol:dalembert} is the %analogue in the semi-discrete setting of the familiar  d'Alembert formula for the solution to the wave %equation in the fully continuous model:
%\begin{equation}
% u(x,t) = \frac{1}{2} (u_0(x+t) +u_0(x-t)) +  \frac{1}{2} \int_{x-t}^{x+t} v_0(s) \; \ud s .
%\label{eq:wave:eqn:sol:dalembert}
%\end{equation}
Our article focuses on the solution in \eqref{eq:solution}.

To connect this solution \eqref{eq:solution} to the matrix function point of view let us make the particular choice of function
 \begin{equation}
 f(z) \equiv \cos \left( t \frac{\sqrt{z} }{ \Delta x} \right).
 \end{equation}
Define
\begin{equation}
 \bm{T} \equiv \sum_{k=1}^N f(\lambda_k) \bm{T}_k
\label{eq:Twave:T}
\end{equation}
 with $\bm{T}_k$ as in \eqref{eq:T:part} but now $h = \Delta x = \frac{2}{N-1}$, 
  and 
\begin{equation}
  \bm{H} \equiv \sum_{k=1}^N f(\lambda_k) \bm{H}_k,
  \label{eq:Hwave:H}
\end{equation}
   with $ \bm{H}_k$ as in \eqref{eq:H:part}.
The spacing $\Delta x = \frac{2}{N-1}$ is chosen so that the mesh consists of $N$ equally-spaced points, including the endpoints -1 and 1.
 Now  via \eqref{eq:diagonalize:TplusH},  the solution \eqref{eq:solution} to the wave equation is the sum of a Toeplitz wave, $\bm{T} \bm{u}(0)$,  and a Hankel wave, $\bm{H} \bm{u}(0)$.
 That is,
 \begin{equation}
\bm{u}(t) = \;  f(\bm{K}) \bm{u}(0) \; =  \;\;\;\;  \underbrace{\bm{T} \bm{u}(0)}_\text{Toeplitz wave} \;\;\;\; + \;\;\;\;   \underbrace{\bm{H} \bm{u}(0).}_\text{Hankel wave} %u(t) = e^{ \pm i \sqrt{K} t / h} u(0)  .
\label{eq:Twave:plus:Hwave:solution}
\end{equation}
  
Next, we seek to  to obtain  explicit expressions for the Toeplitz  and Hankel waves in terms of certain finite sums of Bessel functions of the first kind.

  \section{Expressions for the Toeplitz and Hankel waves as certain sums of Bessel functions}
\label{sec:4}

In this section we derive explicit representations of the Toeplitz and Hankel waves on a symmetric one dimensional domain with a symmetric initial condition. We will show that the nonconstant Toeplitz wave traverses the domain on odd traversals of the domain (the first traversal is a half traversal), while the nonconstant Hankel wave appears on even traversals. These explicit expressions are given as a finite sum of even Bessel functions of the first kind and the number of traversals appears as a parameter on the total number of Bessel functions required. In order to construct the Toeplitz and Hankel waves we need some preliminary lemmas and definitions, some of which we collect from \cref{sec:2}, but here we specialise the expressions to a more convenient form. 

%\vspace{1cm}

%\noindent
%\textbf{Lemma 1.}
\begin{lemma}
Let $\bm{K} = \textrm{tridiag}(-1, 2, -1)$ be the $N \times N$ tridiagonal matrix with 2 on the diagonal and -1 on the first upper and lower subdiagonals as displayed in \eqref{eq:K}; then the eigenvalues $\lambda_k, \> k=1, \cdots, N$ of $\bm{K} $ satisfy
\begin{equation}
\sqrt{\lambda_k} = 2 \sin \frac{k \pi}{2(N+1)}, \quad k = 1,\cdots,N.
\label{eq:eq1}
\end{equation}
\end{lemma}

%\noindent
%\textbf{Proof.}
\begin{proof}
From \eqref{eq:eigenvalues} and with
$$ \lambda_k = 2 - 2 \cos \frac{k \pi}{N+1} = 2(1 - \cos \frac{k \pi}{N+1}) = 4 \sin^2 \frac{k \pi}{2(N+1)},$$
 the result follows. 
 %$\square$
\end{proof}

%\vspace{1cm}
 
%\noindent
%\textbf{Definition 2.}
% For $k = 1, \cdots, N$ let $T_k$ and $H_k$ denote the Toeplitz and Hankel matrices of size $N$ where %the $(m,n)$ elements are
%\begin{eqnarray*}
%(T_k)_{mn} &=& \frac{1}{N+1} \cos \left( (m-n) \frac{k \pi}{N+1}\right), \quad m,n = 1, \cdots, N,  \\
%(H_k)_{mn} &=& \frac{-1}{N+1} \cos \left( (m+n) \frac{k \pi}{N+1}\right), \quad m,n = 1, \cdots, N.
%\end{eqnarray*}

We now consider the spatially discretised wave equation based on the central difference approximation on $[-1, 1]$ with constant discretisation $\Delta x = \frac{2}{N-1}$, $N$ is odd. Note that in the case $N$ is odd, then 0 is always in the set of mesh points. Hence the discretised spatial operator is $\frac{1}{(\Delta x)^2} \> \bm{K} $ and the travelling wave is given by
\begin{equation}
\bm{u}(t) = f\left( \frac{t}{\Delta x} \sqrt{\bm{K} }\right) \> \bm{u}(0),
\label{eq:eq2}
\end{equation}
where $f$ denotes the  matrix function $\cos \left( \frac{t}{\Delta x} \sqrt{\bm{K} }\right)$. 
We note that the eigenvalues of this matrix $\cos \left( \frac{t}{\Delta x} \sqrt{\bm{K} }\right)$ are
\begin{equation}
f(\lambda_k) = \cos \left( \frac{t}{\Delta x} \sqrt{\lambda_k} \right) = \cos \left( 2 \frac{t}{\Delta x} \sin \frac{k \pi}{2(N+1)} \right ), \quad k = 1, \cdots, N.
\label{eq:eq3}
\end{equation}

We can now define the Toeplitz and Hankel components of this matrix function. They are

\begin{eqnarray}
\textit{TP}(t) &=& \sum_{k=1}^N \cos \left( \frac{t}{\Delta x} \sqrt{\lambda_k} \right) \bm{T}_k = \sum_{k=1}^N \cos \left( 2 \frac{t}{\Delta x} \sin \frac{k \pi}{2(N+1)} \right) \bm{T}_k \nonumber \\
&& \label{eq:eq4} \\
\textit{HK}(t) &=& \sum_{k=1}^N \cos \left( \frac{t}{\Delta x} \sqrt{\lambda_k} \right) \bm{H}_k = \sum_{k=1}^N \cos \left( 2 \frac{t}{\Delta x} \sin \frac{k \pi}{2(N+1)} \right) \bm{H}_k. \nonumber
\end{eqnarray}

Hence the Toeplitz and Hankel waves are
\begin{equation}
T(t) = \textit{TP}(t) \bm{u}_0, \quad H(t) = \textit{HK}(t) \bm{u}_0
\label{eq:eq5}
\end{equation}
where the full wave is $$ \bm{u}(t) = T(t) + H(t).$$
The choice of notation $T(t)$ and $H(t)$ indicates the close relationship to the matrices appearing in \eqref{eq:Twave:T}, \eqref{eq:Hwave:H}, and \eqref{eq:Twave:plus:Hwave:solution}.

In what follows, we will sample time at equally-spaced time points with $t = j \Delta t$, $j$ a positive integer.  We are now discretising in time, but note that we are still evaluating the semi-discrete model (where time remains continuous) exactly at these chosen discrete time points. This is as opposed to what is usually termed a `fully discrete' model, where some additional approximation is introduced.  
%\colorbox{yellow}{\parbox{0.98\textwidth}{
We will also assume without loss of generality that $\Delta t = \Delta x$.
%}}
% as this choice respects the speed of the travelling wave. This is the special case in which the characteristics of the underlying partial differential equation are at an angle of $45^{\circ}$.

Now the Toeplitz and Hankel waves are
\begin{eqnarray}
T(j \Delta t) &=& \sum_{k=1}^N \cos \left( 2j \sin \frac{k \pi}{2(N+1)} \right ) \bm{T}_k \bm{u}_0 \nonumber \\
&& \label{eq:eq6} \\
H(j \Delta t) &=& \sum_{k=1}^N \cos \left( 2j \sin \frac{k \pi}{2(N+1)} \right ) \bm{H}_k \bm{u}_0. \nonumber
\end{eqnarray}

The expressions in \eqref{eq:eq6} lead us to consider Bessel functions $J_l(t)$ of the first kind. 
%The following lemma gives some important properties of these Bessel functions.
In the following, we list some lemmas which yield some important properties of the Bessel functions.

%\vspace{1cm}
%\noindent
%\textbf{Lemma 3.}
\begin{lemma}
The Bessel functions of the first kind, $\alpha$ real, are
$$ J_{\alpha} (t) = \sum_{m=0}^\infty \frac{(-1)^m}{m! \> \Gamma(1+\alpha+m)} \left( \frac{t}{2} \right)^{\alpha+2m}.$$
For integer values of the index $\alpha$, the following properties hold.
\begin{eqnarray*} 
J_n(t) &=& \frac{1}{\pi} \int_0^\pi \cos(ns-t \sin(s)) ds. \\
\int_0^\infty \frac{1}{t} J_n(t) J_l(t)dt &=& \frac{2}{\pi} \frac{\sin(\frac{\pi}{2}(n-l))}{n^2 - l^2}. \\
\mu^{-\alpha} J_\alpha(\mu t) &=& \sum_{m=0}^\infty \frac{1}{m!} \left( \frac{(1-\mu^2)t}{2} \right)^m J_{\alpha+m}(t).
\end{eqnarray*}
\end{lemma}
\noindent
%\colorbox{yellow}{\parbox{0.98\textwidth}{
For non-negative integers $n$, $J_n$ has an infinite number of zeros and Siegel's Theorem states that for any integers $n \geq 0$ and $p \geq 1$, $J_n$ and $J_{n+p}$ have no common zeros other than $t=0$. (See, for example, Abramowitz and Stegun \cite{Abram}).
%The Lemma that is fundamental to our analysis is the well-known Bessel relation given in \cref{lemma:4}.
The following \Cref{lemma:4} gives us the well-known Bessel relation, which will be fundamental for our analysis.
%}}
\vspace{1cm}
\noindent
%\textbf{Lemma 4.}
\begin{lemma}
\label{lemma:4}
The Bessel relation is given by
$$ \cos(t \, \sin x) = J_0(t)+ 2 \sum_{l=1}^\infty J_{2l}(t) \, \cos(2lx). $$
\end{lemma}

As a consequence of \Cref{lemma:4},
\begin{equation}
\cos(2j\, \sin \frac{k \pi}{2(N+1)}) = J_0(2j) + 2 \sum_{l=1}^\infty J_{2l}(2j)\, \cos(l \, \frac{k \pi}{N+1}).
\label{eq:eq7}
\end{equation}

It is clear that we can use \eqref{eq:eq7} in simplifying the expressions for the Toeplitz and Hankel waves in \eqref{eq:eq6}. In particular we can rewrite \eqref{eq:eq6} as
\begin{equation*}
T(j \Delta t) = J_0(2j) \left( \sum_{k=1}^N \bm{T}_k \right) \, \bm{u}_0 
+ 2 \sum_{l=1}^\infty J_{2l}(2j) \, \left( \sum_{k=1}^N \cos \left(\frac{kl \pi}{N+1} \right) \> \bm{T}_k \right) \> \bm{u}_0.
%\label{eq:eq8}
\end{equation*}
\begin{equation*}
H(j \Delta t) = J_0(2j) \left( \sum_{k=1}^N \bm{H}_k \right) \, \bm{u}_0 
+ 2 \sum_{l=1}^\infty J_{2l}(2j) \, \left( \sum_{k=1}^N \cos \left(\frac{kl \pi}{N+1} \right) \> \bm{H}_k \right) \> \bm{u}_0.
%\label{eq:eq9}
\end{equation*}

Using the definition of the Toeplitz and Hankel matrices in \eqref{eq:T:part} and \eqref{eq:H:part} and the trigonometric relation
$$ 2 \cos \theta_1 \cos \theta_2 = \cos(\theta_1+\theta_2) + \cos(\theta_1-\theta_2),$$
then we can write that component $m$, $m = 1, \cdots, N$ of the vectors $T(j \Delta t)$ and $H(j \Delta t)$ are 
\begin{eqnarray}
(N+1) T_m (j \Delta t) &=& J_0(2j) \sum_{n=1}^N \sum_{k=1}^N \cos \left( (m-n) \frac{k \pi}{N+1} \right) \, \bm{u}_0 \nonumber  \\
  &+& \sum_{l=1}^\infty J_{2l}(2j) \sum_{n=1}^N \sum_{k=1}^N \left( \cos \left( (l+m-n) \frac{k \pi}{N+1} \right) \right.  \nonumber \\
&+& \left. \cos \left( (l-(m-n))  \frac{k \pi}{N+1} \right) \right) \, \bm{u}_0  \label{eq:eq10}
\end{eqnarray}
\begin{eqnarray}
-(N+1) H_m (j \Delta t) &=& J_0(2j) \sum_{n=1}^N \sum_{k=1}^N \cos \left( (m+n) \frac{k \pi}{N+1} \right) \, \bm{u}_0 \nonumber \\
  &+& \sum_{l=1}^\infty J_{2l}(2j) \sum_{n=1}^N \sum_{k=1}^N \left( \cos \left( (l+m+n) \frac{k \pi}{N+1} \right) \right. \nonumber \\
& +& \left.  \cos \left( (l-(m+n))  \frac{k \pi}{N+1} \right) \right) \, \bm{u}_0.   \label{eq:eq11}
\end{eqnarray}
%\colorbox{yellow}{\parbox{0.98\textwidth}{
A reminder about notation may be helpful here: $T_m (t)$ denotes component $m$ of the time-varying vector that is the Toeplitz wave at time $t$ that was introduced in \eqref{eq:eq6}, whereas $\bm{T}_k$ is the matrix  defined  in \eqref{eq:T:part}.
%}}

%\noindent \colorbox{yellow}{\parbox{0.98\textwidth}{
In order to make further progress we use the following lemma, which is known as the Lagrange identity, and a subsequent lemma.
%}}

%\colorbox{yellow}{\parbox{0.98\textwidth}{
%\vspace{1cm}
%\noindent
%\textbf{Lemma 5.}
\begin{lemma}
%The Lagrange identity.
Let 
$$L(\theta) := \sum_{k=1}^N \cos k \theta. $$
If the angle is not an integer multiple of $2 \pi$, i.e. $\theta \neq 2 \rho \pi$, $\rho$ an integer, then
 $$
 L(\theta) = -\frac{1}{2} + \frac{\sin((N+\frac{1}{2}) \theta)}{2 \sin \frac{\theta}{2}}.$$
Otherwise, if $\theta = 2 \rho \pi$, then  $ L(\theta)= N$.
\end{lemma}
%\noindent
%\textbf{Proof.}
%}}
\begin{proof}
Consider $\sum_{k=1}^N e^{i k \theta}$. Then use De Moivre's Theorem and equate real and imaginary parts. 
%$\square$
\end{proof}

%\vspace{1cm}
%\noindent
%\textbf{Lemma 6.}
%\colorbox{yellow}{\parbox{0.98\textwidth}{
\begin{lemma}
\label{lemma:6}
For $\theta = p \frac{\pi}{N+1}$, with $p$ an integer, the following identity holds true:
$$ L(\theta) = \left\{ \begin{array}{rcl} N & , & p = 2 \rho (N+1), \quad \rho = 0, 1, 2, \cdots \\
0 & , & p \textrm{ is odd} \\
-1 & , & p \textrm{ is even (and } \quad p \ne 2 \rho (N+1), \quad \rho = 0, 1, 2, \cdots \textrm{)}
\end{array} \right. $$
\end{lemma}
%}}

%\noindent
%\textbf{Proof.}
\begin{proof}

If $\theta = p \frac{\pi}{N+1}$ with $p = 2 \rho (N+1)$, then $\theta = 2 p \rho \pi$ and clearly $L(\theta) = N$. Otherwise using the relationship
$$ \sin (A-B) = \sin A \cos B - \cos A \sin B$$
gives
\begin{eqnarray*}
L(\theta) &=& -\frac{1}{2} + \frac{\sin ((N+1) \theta - \frac{\theta}{2})}{2 \sin \frac{\theta}{2}} \\
&=& -\frac{1}{2} + \sin p \pi \frac{\cos \frac{\theta}{2}}{2 \sin \frac{\theta}{2}} - \frac{1}{2} \cos p \pi \\
&=& -\frac{1}{2}(1+\cos p \pi). \quad \quad 
%\square
\end{eqnarray*}
which is $0$ or $-1$ according to the parity of $p$.
\end{proof}

We will use \Cref{lemma:6} to simplify the Toeplitz wave and the Hankel wave formulation given in \eqref{eq:eq10} and \eqref{eq:eq11}. In the case of the Toeplitz wave, $\theta$ can take on the values $(m-n) \frac{\pi}{N+1}$, $(l+m-n) \frac{\pi}{N+1}$, $(l-(m-n)) \frac{\pi}{N+1}$ so $p = m-n, \> l+m-n$ or $l-(m-n)$. 
Therefore we must take care when these values of $p$ are of the form $p = 2 \rho (N+1)$, $\rho = 0, 1, 2, \cdots$. Similarly, in the case of the Hankel wave, the corresponding $p$ values are $m+n$, $l+m+n$ or $l-(m+n)$.

\noindent %\colorbox{yellow}{\parbox{0.98\textwidth}{
We now introduce a parameter $R \in \{0,1,2, \ldots, \}$. 
The integer  $R$ will be used as an upper limit for the number of terms in a sum, in the expressions below.
We will separate the cases in which the indices take the values $p = 2 \rho (N+1)$, $\rho = 0, 1, \cdots, R-1$, from the other values of $p$. 
%}}

We note $m-n$ can take on values $k$, say, where $k = -(N-1),\cdots,-1,$ $0,$ $1, \cdots, (N-1)$. 
Hence for $\rho=0,1,\ldots$ with $p = 2 \rho (N+1)$ then
$$2l = 4 (N+1) \rho \mp 2k$$
and the corresponding $L(\theta) = N$ in these cases. We will now simplify the expression $(N+1) T(j \Delta t)$ in vector form.

\noindent %\colorbox{yellow}{\parbox{0.98\textwidth}{
If $m-n = 0$ then there is a component of the Toeplitz wave 
$$N \left(J_0(2j) + 2 \sum_{\rho = 1}^{R-1} J_{4(N+1) \rho} (2j) \right) \, \bm{u}_0.$$
Similarly, if $|m-n| = k$, $k = 1, \cdots,N-1$, then we can use the shift matrices $E_1, \cdots, E_{N-1}$ where $E_k$ introduced earlier.
%}}

Introducing the vectors $\bm{\nu}_k$, $k = 0, \cdots, N-1$, where 
\begin{eqnarray*}
\bm{\nu}_0 &=& \bm{u}_0 \\
\bm{\nu}_k &=& (E_k + E_k^\top)\, \bm{u}_0, \quad k = 1, \cdots,N-1 
\label{eq:eqEk:shift:matrix}
\end{eqnarray*}
leads to the first part of the Toeplitz wave as
\begin{equation}
F = N \left( \left( J_0(2j) + 2 \sum_{\rho=1}^{R-1} J_{4(N+1)\rho + 2k} (2j) + J_{4(N+1)\rho-2k}(2j) \right) \, \bm{\nu}_k \right).
\label{eq:eq12}
\end{equation}

We will write the factor at the front as $(N+1\,-1)$. Subtracting off this term and reincorporating into \eqref{eq:eq10} and then using Lemma 6 with $L(\theta)$ either 0 or -1 and finally dividing throughout by $N+1$ leads to
\begin{eqnarray}
T(j \Delta t) &=& (J_0(2j) + 2 \sum_{\rho=1}^{R-1} J_{4(N+1)\rho}(2j) )\, \bm{\nu}_0 \label{eq:eq13}  \\
&+& \sum_{k=1}^{N-1} \left( J_{2k}(2j) + \sum_{\rho=1}^{R-1} (J_{4(N+1)\rho+2k}(2j) + J_{4(N+1)\rho -2k}(2j)) \right)\, \bm{\nu}_k \nonumber  \\
&-& X, \nonumber 
\end{eqnarray}
where
\begin{equation}
(N+1)\,X = (J_0(2j) \, A + 2 \sum_{l = 1}^\infty J_{4l}(2j)\, A + 2 \sum_{l = 1}^\infty J_{4l - 2}(2j)\, B)\, \bm{u}_0,
\label{eq:eq14}
\end{equation}
where $A$ and $B$ are the $N \times N$ matrices (as an aside, notice that these matrices have the special property that they are simultaneously both Toeplitz and Hankel)
\begin{equation}
A = \left( \begin{array}{cccccc}
1&0&1&0&1&\cdots \\
0&1&0&1&0&\cdots \\
1&0&1&0&1&\cdots \\
\vdots & \vdots & \vdots & \vdots & \vdots & \ddots \\ \end{array}
\right), \quad \quad 
B = \left( \begin{array}{cccccc}
0&1&0&1&0&\cdots \\
1&0&1&0&1&\cdots \\
0&1&0&1&0&\cdots \\
\vdots & \vdots & \vdots & \vdots & \vdots & \ddots \\ \end{array}
\right). 
\label{eq:A:B:simultaneous:TplusH}
\end{equation}

\noindent  %\colorbox{yellow}{\parbox{0.98\textwidth}{
The expression for the term $X$ above in \cref{eq:eq14} is complicated.
It will be helpful to simplify this expression for $X$ because this will simplify the expression we obtain later for the Toeplitz wave.
This simplification is made possible by  \Cref{lemma:7} and \Cref{lemma:8}, as follows.
%}}
%\vspace{1cm}
%\noindent 
%\textbf{Lemma 7.}
\begin{lemma}
\label{lemma:7}
The following relationships hold:
\begin{eqnarray*}
J_0(t) + 2 \sum_{L=1}^\infty J_{4l}(t) &=& \frac{1}{2} (1 + \cos t) \\
2 \sum_{l = 1}^\infty J_{4l-2}(t) &=& \frac{1}{2}(1 - \cos t).
\end{eqnarray*}
\end{lemma}

%\noindent
%\textbf{Proof:}

%\colorbox{yellow}{\parbox{0.98\textwidth}{
\begin{proof}
From \Cref{lemma:4}, the Bessel relations, we have
\begin{eqnarray*}
\cos(t \sin \pi) &=& 1 = J_0(t) + 2 \sum_{l=1}^\infty J_{4l-2}(t) + 2 \sum_{l=1}^\infty J_{4l}(t), \\
\cos(t \sin \frac{\pi}{2}) &=& \cos t = J_0(t) - 2 \sum_{l=1}^\infty J_{4l-2}(t) + 2 \sum_{l=1}^\infty J_{4l}(t).
\end{eqnarray*}
The relationships claimed in the lemma are obtained by addition and subtraction of these two equalities, and noticing cancelation of terms. 
%$\square$
\end{proof}
%}}

%\noindent \colorbox{yellow}{\parbox{0.98\textwidth}{
We can now simplify $X$ by applying \Cref{lemma:7}.
%}}
%\vspace{1cm}
%\noindent
%\textbf{Lemma 8.}
\begin{lemma}
\label{lemma:8}
$$ X = \frac{1}{2(N+1)} (\alpha \bm{e} + \beta \cos(2j)\,\bm{w}),$$
where $\bm{e} = (1,1,\cdots,1)^\top, \> \bm{w} = (1,-1,1,-1,\cdots,1)^\top, \> \alpha = \bm{e}^\top \bm{u}_0, \> \beta = \bm{w}^\top \bm{u}_0$.
\end{lemma}
%\noindent
%\textbf{Proof:}
\begin{proof}
From \eqref{eq:eq14} and \Cref{lemma:7}
\begin{eqnarray*}
(N+1)\,X &=& \frac{1}{2} (1+\cos 2j)\,A\,\bm{u}_0 + \frac{1}{2}(1-\cos 2j)\,B\,\bm{u}_0 \\
&=& \frac{1}{2}(A+B)\,\bm{u}_0 + \frac{1}{2} \cos 2j (A-b)\,\bm{u}_0 \\
&=& \frac{1}{2}\,\bm{e}\,\bm{e}^\top\,\bm{u}_0 + \frac{1}{2} \cos 2j\,\bm{w}\,\bm{w}^\top\,\bm{u}_0,
\end{eqnarray*}
and the result is proved. 
%$\square$
\end{proof}

%\noindent \colorbox{yellow}{\parbox{0.98\textwidth}{
We can now write the characterisation of the Toeplitz wave over the $R$ traversals in two ways: either in terms of the vectors $\bm{\nu}_k$ (\Cref{theorem:1}) or increasing values of the Bessel indices (\Cref{corollary:1}).
A reason to seek results such as these is that they allow us to express the wave as a sum over the number of traversals, which is useful later in computer simulations when we want to control the number of reflections, for example.
%}}

%\vspace{1cm}
%\noindent
%\textbf{Theorem 1.}
\begin{theorem}
\label{theorem:1}
The Toeplitz wave over $R$ traversals is
\begin{eqnarray}
T(j \Delta t) &=& (J_0(2j) + 2 \sum_{\rho=1}^{R-1} J_{4(N+1)\rho}(2j) )\, \bm{\nu}_0 \nonumber  \\
&+& \sum_{k=1}^{N-1} \left( J_{2k}(2j) + \sum_{\rho=1}^{R-1} (J_{4(N+1)\rho+2k}(2j) + J_{4(N+1)\rho -2k}(2j)) \right)\, \bm{\nu}_k \nonumber  \\
&-& \frac{1}{2(N+1)}\,(\alpha \, \bm{e} + \beta \cos(2j)\,\bm{w}),  \label{eq:eq15} 
\end{eqnarray}
where $\alpha = \bm{e}^\top \bm{u}_0,$ $\beta = \bm{w}^\top \bm{u}_0,$ $\bm{e} = (1,\cdots,1)^\top,$ $\bm{w} = (1,-1,\cdots,1)^\top,$ $\bm{\nu}_0=\bm{u}_0,$ and $\bm{\nu}_k = (E_k + E_k^\top)\,\bm{u}_0,\> k=1,\cdots,N-1.$
\end{theorem}
We can rewrite \eqref{eq:eq15} in the ascending index of the Bessel function.

%\vspace{1cm}
%\noindent
%\textbf{Corollary 1.}
\begin{corollary}
\label{corollary:1}
The Toeplitz wave over $R$ traversals is
\begin{eqnarray}
T(j \Delta t) & = & \sum_{l=0}^{N-1} J_{2l}(2j)\,\bm{\nu}_l  +  \sum_{\rho=1}^{R-1} \left( \sum_{l=0}^{N-1} J_{4(N+1)\rho-2l}(2j)\,\bm{\nu}_{l} \right. \nonumber \\
& +& \left.  \sum_{l=0}^{N-1} J_{4(N+1)\rho+2l}(2j)\,\bm{\nu}_l \right) \nonumber \\
& - & \frac{1}{2(N+1)} (\alpha \bm{e} + \beta \cos(2j)\,\bm{w}). \label{eq:eq16}
\end{eqnarray}
\end{corollary}

We can observe from \eqref{eq:eq16} that the $J_{4(N+1)\rho}$ term is repeated twice and there are always three missing terms when moving to the next traversal - e.g. $J_{2N}$, $J_{2(N+1)}$, $J_{2(N+2)}$ or $J_{6N+4}$, $J_{6N+6}$, $J_{6N+8}$, etc.

We can now construct the Hankel wave based on \eqref{eq:eq11}. We must consider the cases where $p = 2\rho(N+1), \> \rho = 0,\cdots,R-1$ where now $p = m+n, l+m+n$ or $l-(m+n)$. Let $m+n = k$; then $k$ can take values $2, 3, \cdots,2N$. Hence
$$ l \pm k = 2 \rho (N+1)$$
so
\begin{equation}
2l = 4(N+1)\rho \mp 2k, \quad k = 2, \cdots,2N, \quad \rho = 0, \cdots,R-1.
\label{eq:eq17}
\end{equation}
For these values $L(\theta) = N$.

Let $F_1, \cdots, F_{2N-1}$ be the $N \times N$ Hankel matrices in which 1's are on the $k$th backwards subdiagonal of $F_k$, and all other entries are zero. 
That is,
$$F_1 = \left( \begin{array}{cccc} 1 & 0 & \cdots & 0 \\  0 &  0 & \cdots & 0 \\  & & \cdots & \\ 0 & 0 & \cdots & 0  \end{array} \right), \quad F_2 =  \left( \begin{array}{ccccc} 0 & 1 & 0 & \cdots & 0 \\ 1 & 0 & 0 & \cdots & 0 \\  &&& \cdots & \\ 0 & 0 & 0 & \cdots & 0  \end{array} \right), \cdots, $$
\begin{equation}
 F_N = \left( \begin{array}{cccc} 0 & \cdots & 0 & 1 \\ 0 & \cdots & 1 & 0 \\  & \iddots && \\ 1 & \cdots & 0 & 0 \end{array} \right), \cdots,
 F_{2N-1} = \left( \begin{array}{ccccc} 0 & 0 & \cdots & 0 & 0 \\  0 & 0 & \cdots & 0 & 0 \\ && \cdots &&  \\  0 & 0 & \cdots &0 & 1  \end{array} \right).
 \label{eq:backward:hanel:shifts}
 \end{equation}
 Then we can construct the vector $H(j \Delta t)$ for these values of $l$ given in \eqref{eq:eq17}. It is easy to show that this leads to the term
$$ -\sum_{k=1}^{2N-1} (F_k + F_{2n-k})\, \sum_{\rho = 0}^{R-1} J_{4\rho(N+1)+2k+2}(2j)\,\bm{u}_0.$$
By the same argument as for the Toeplitz wave, the remaining component is
$$ \frac{1}{2(N+1)} (\alpha \bm{e} + \beta \cos(2j)\,\bm{w}).$$
This leads to~\Cref{theorem:2}.

%\vspace{1cm}
%\noindent
%\textbf{Theorem 2.}
\begin{theorem}
\label{theorem:2}
The Hankel wave over $R$ traversals is
\begin{eqnarray}
H(j \Delta t) &=& \frac{1}{2(N+1)} (\alpha \bm{e} + \beta \cos(2j)\,\bm{w}) \nonumber \\
&-& \sum_{k=1}^{2N-1} (F_k + F_{2n-k})\, \sum_{\rho=0}^{R-1} J_{4\rho(N+1)+2k+2}(2j)\,\bm{u}_0.
\label{eq:eq18}
\end{eqnarray}
\end{theorem}

%\noindent \colorbox{yellow}{\parbox{0.98\textwidth}{
In a way that is analogous to what we did for the Toeplitz wave,  we 
rewrite \eqref{eq:eq18} for the Hankel wave in the ascending index of the Bessel function.
%}}

%\vspace{1cm}
%\noindent
%\textbf{Corollary 2.}
\begin{corollary}
The Hankel wave over $R$ traversals is 
\begin{eqnarray}
H(j \Delta t) &=& \frac{1}{2(N+1)}\, (\alpha \bm{e} + \beta \cos(2j) \,\bm{w}) \nonumber \\
&-& \sum_{\rho=0}^{R-1}\, \sum_{l=1}^{2N-1} J_{4\rho(N+1)+2l+2}(2j)\,\bm{\psi}_l, \label{eq:eq19}
\end{eqnarray}
where
$$\bm{\psi}_l = (F_l + F_{2N-l})\,\bm{u}_0, \quad l = 1, \cdots,N-1$$
and so $\bm{\psi}_{2N-l} = \bm{\psi}_l, \quad l = 1, \cdots,N$.
\end{corollary}

In the case of the Hankel wave we can see that similarly to the Toeplitz wave there are three missing terms when moving to the next traversal - e.g. $J_{4N+2}$, $J_{4N+4}$, $J_{4N+6}$, but the indices are shifted by $2(N+1)$ from the Toeplitz case, which of course corresponds to a shift by the length of the domain.

\section{Simulations: The reflectionless properties of the Toeplitz and Hankel waves}
\label{sec:5}
%\noindent \colorbox{yellow}{\parbox{0.98\textwidth}{
We can now perform computer simulations of the semi-discrete model \eqref{eq:wave:semi:discrete} by computing the solution in  \eqref{eq:solution}.
That solution  \eqref{eq:solution} and thus a simulation  are computed via  the wave-kernel matrix function software of Nadukandi \& Higham \cite{WaveKernelsHigham}.
For reference, the solution coming from d'Alembert's formula (as if the problem were on the whole real line with no boundaries) is also included.
Separately, we can also simulate the Toeplitz waves in two different ways, using either \eqref{eq:Twave:T} or the Bessel expansions \eqref{eq:eq15}.
Likewise, we can simulate the Hankel waves in two different ways, using either \eqref{eq:Hwave:H} or \eqref{eq:eq18}. 
The sum of the Toeplitz wave and the Hankel wave is always exactly equal to the solution to the wave equation with Dirichlet conditions, as in \eqref{eq:Twave:plus:Hwave:solution}, and this property is consistent with results of the simulations, as expected.
%}}
  
 Consider the simulation in  \cref{fig:wave}.
 The simulation begins with a smooth symmetric Gaussian, centered on $x=0$ as an initial condition.
 \Cref{fig:wave} shows the solution at three time-points: before, during, and after the first moment at which the wave reaches the boundary. 
 The following points are notable.
 % \colorbox{yellow}{\parbox{0.98\textwidth}{
 \begin{itemize}
\item \textit{Before} the wave reaches the boundary (this corresponds to $t \leq 1$), the solution is ``purely Toeplitz'' and the Hankel part of the solution is a spatial and temporal constant that is nearly zero -- this constant is discussed in \cref{sec:discussion}.
\item  \textit{After} the wave reaches the boundary, the opposite is true: the solution is ``purely Hankel,'' and the Toeplitz part of the solution is a constant.
(And this constant is of the opposite sign to the constant in the dot point above.) 
This corresponds to $1 \leq t \leq 3$.
%This constant is of the opposite sign to the constant Hankel part of the solution in the `before' panel.

\item If the simulation is run for larger time values, then this behaviour is  repeated, in the sense that on any even traversal the Toeplitz wave is constant, while on any odd traversal the Hankel wave is constant.  
\end{itemize}
%}}

%For a fixed and finite $N$, and increasing time $t$, this description is only approximate because the exact solution to the semi-discrete problem is not constant and becomes more oscillatory for large $t$. 

 Naively, one might be tempted to incorporate the  condition \eqref{eq:orr-sommerfeld}  into simulations in the semi-discrete setting, in the hope of achieving a reflectionless boundary condition.
 Unfortunately this leads to disappointment;  such a simulation does show reflections at the boundary (albeit small reflections that can be reduced by refining the discretization).
 The issue is that the condition \eqref{eq:orr-sommerfeld} is only reflectionless in the fully continuous setting, whereas in computer simulations, there must be some type of discretization (for example, in the semi-discrete setting, we might try a finite difference approximation).
 With this in mind, the reflectionless behaviour that can be achieved with the Toeplitz and Hankel waves on even or odd traversals of the domain  (\cref{fig:wave}) for which  we have derived exact methods of simulation,  and an accompanying analysis to explain the behaviour, is especially satisfying.
 
 %\colorbox{yellow}{\parbox{0.98\textwidth}{
In \cref{fig:wave2} we show  the Toeplitz wave (top left) computed via \eqref{eq:Twave:T} or \eqref{eq:eq15} with $R = 3$ and $t = 3.7$ on the third traversal, and the Hankel wave computed via \eqref{eq:Hwave:H} or \eqref{eq:eq18} with $R = 3$ and $t = 5.7$ on the fourth traversal (top right). 
The good agreement that we see numerically between these two different methods of computing the solution is a check that the theorems and formulas that we present are correct.
The bottom left panel shows the Toeplitz wave  computed via \eqref{eq:Twave:T} or \eqref{eq:eq15} with $R = 3$ and $t = 2 \pi R + 1$.
This purpose of this example is to illustrate that the formulas \eqref{eq:eq15} we derive are a constant, near zero, for all $t > 2 \pi R + 1$, whereas the exact solution to the semi-discrete problem may be complicated for large $t$.
This panel also showcases our ability to control and choose in advance the number (through the choice of $R$) of  reflections in a simulation.
The effect of reducing the spatial discretization error can be seen by comparing the top right panel ($N=301$) with the bottom right panel ($N=1001$).
(Note that our formulas for the semi-discrete problem are exact, and the computations of the solution to the semi-discrete problem shown here are so accurate that they may be regarded as exact solutions.  
The phrase discretization error here is used to indicate the comparison of the solution of the continuous problem to the semi-discrete solution.)
With a smaller $N$, and thus a larger dicretization error,
we see that the semi-discrete problem becomes more oscillatory.
%}} 

%\colorbox{yellow}{\parbox{0.98\textwidth}{
The integer $R \in \{0,1,2,\ldots \}$  appears in our formulas \eqref{eq:eq15} and \eqref{eq:eq18} as the number of terms in a sum.
That number $R$ can be chosen by the user.
The figures illustrate that $R$ can be thought of as the number of full wave traversals of the whole domain.
For example, the choice $R=0$ would allow simulation of the top panel of  \cref{fig:wave}, starting from an  initial condition that is symmetric about $x=0$, and simulating  the first half traversal of the domain. 
Whereas to use the formulas to simulate the bottom panel of  \cref{fig:wave}, would require choosing $R \ge 1$ .
%}}

%\colorbox{yellow}{\parbox{0.98\textwidth}{
Two benefits that our analysis brings to these simulations are noteworthy.
First, we are able to choose in advance the number of reflections.
For example, we could ask for a simulation that allows exactly two reflections, by choosing $R=2$.
In such a simulation, after the second reflection, the Toeplitz waves and the Hankel waves (computed using our formulas that involve summation to an upper limit $R$) vanish for all larger values of time.
This is not possible with existing methods of simulation, such as via the software for the wave-kernel matrix functions \cite{WaveKernelsHigham} that we have also included for comparison here.
Second, we are able to explain why the Toeplitz wave is reflection-less (on even traversals of the domain) in simulations. 
For example, we are able to explain why the Toeplitz wave is perfectly reflection-less the very first time it hits the boundary. 
This would not be possible without the expressions for the Toeplitz wave that we derived.
%}}

 \begin{figure}
\centering
\begin{tabular}{c}
\includegraphics[scale=0.4]{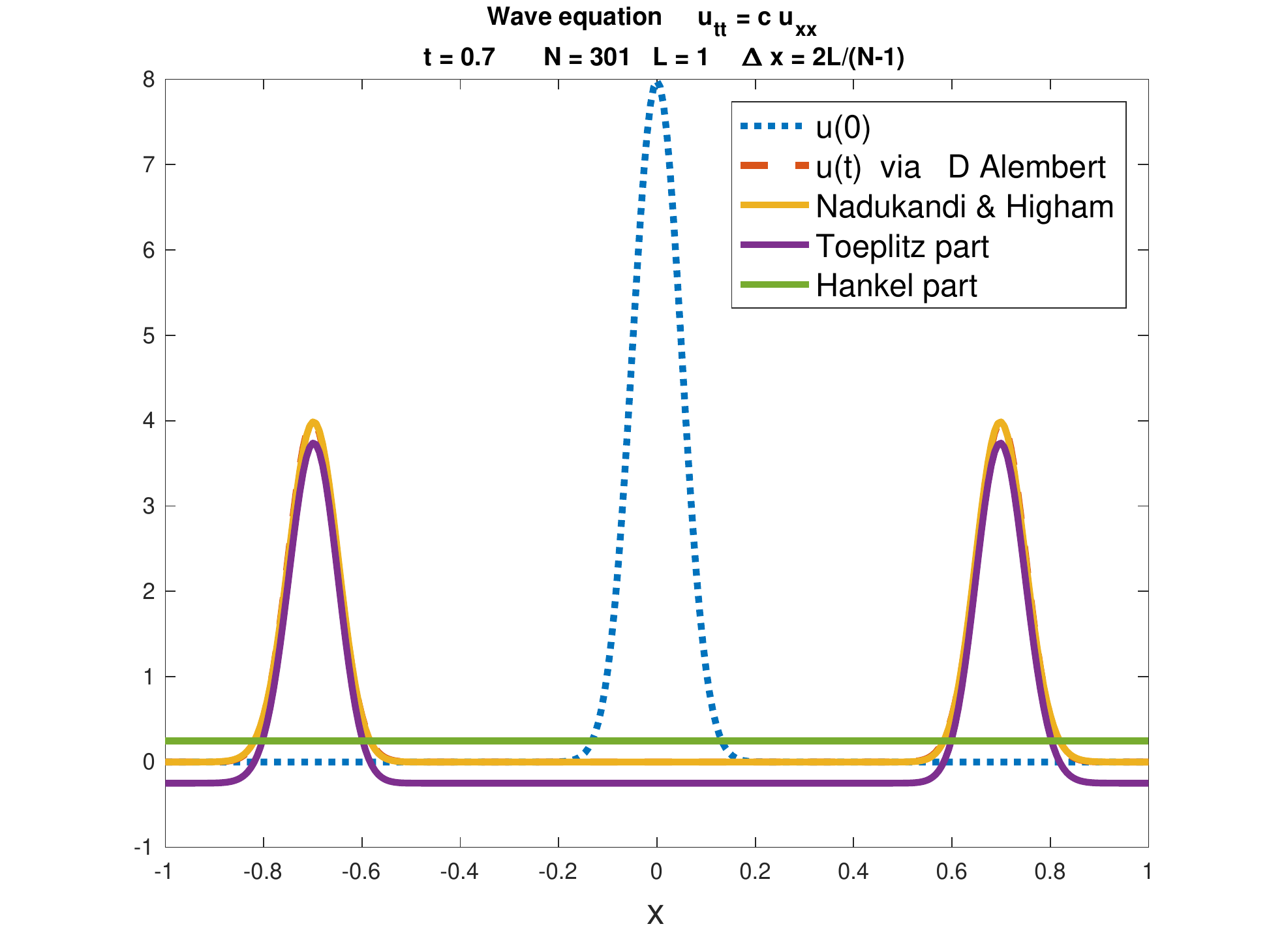}  \\
\includegraphics[scale=0.4]{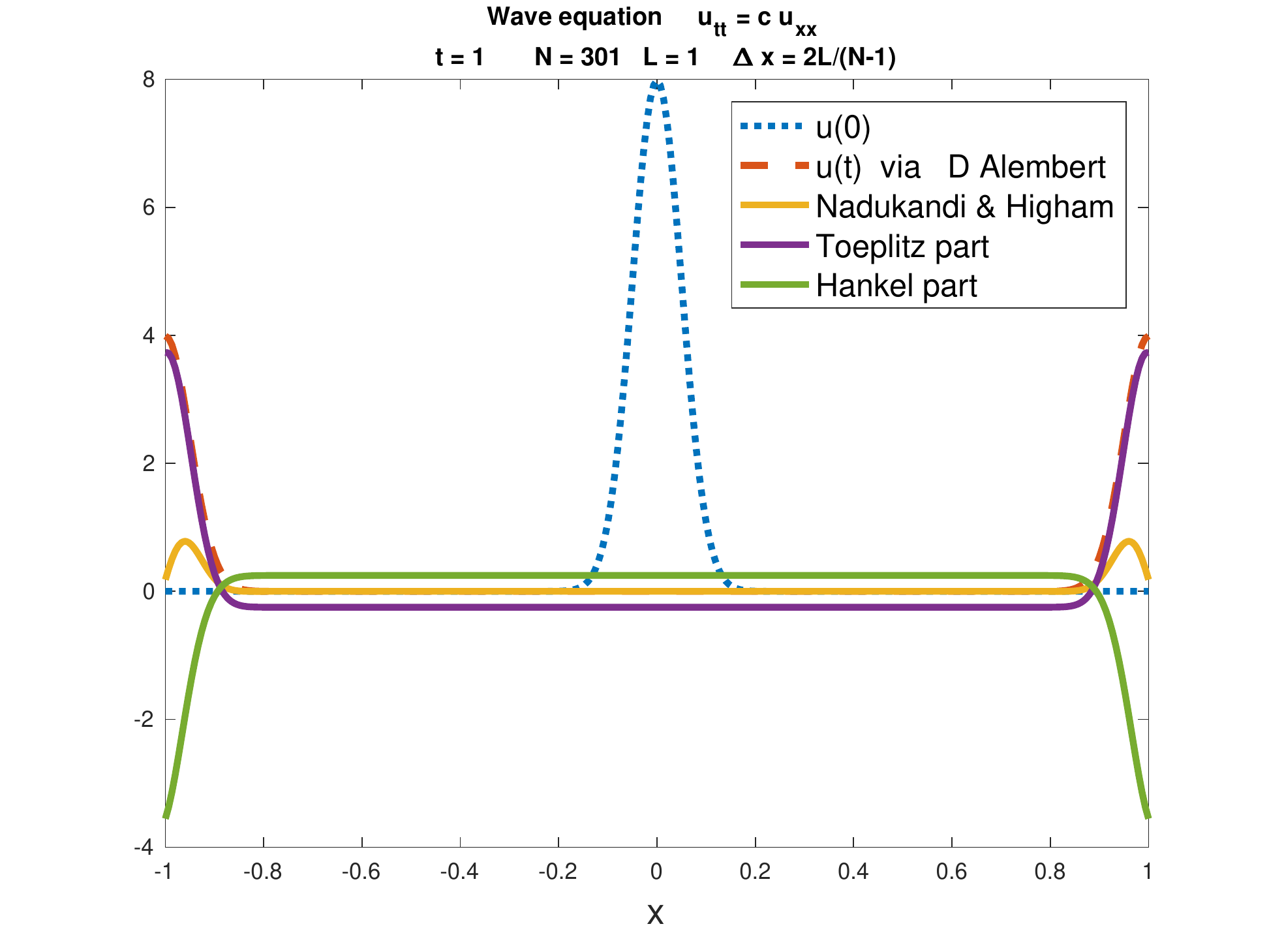} \\
\includegraphics[scale=0.4]{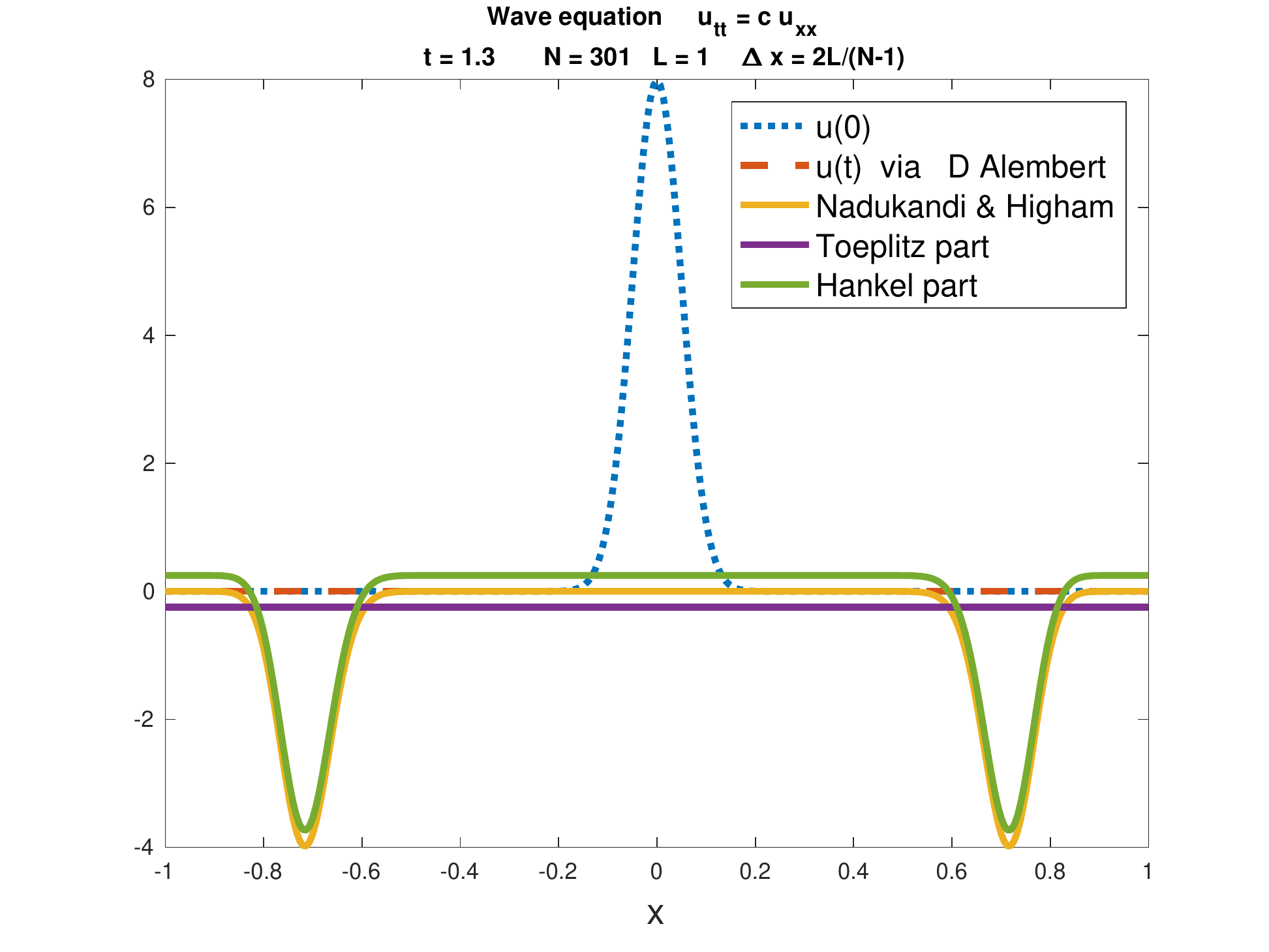}
\end{tabular}
\caption{Before (top), during (middle) and after (bottom) the wave first reaches the boundary.
The Toeplitz wave is  reflectionless on the even traversals.
}
\label{fig:wave}
\end{figure}

\begin{figure}
\centering
\begin{tabular}{cc}
\includegraphics[scale=0.35]{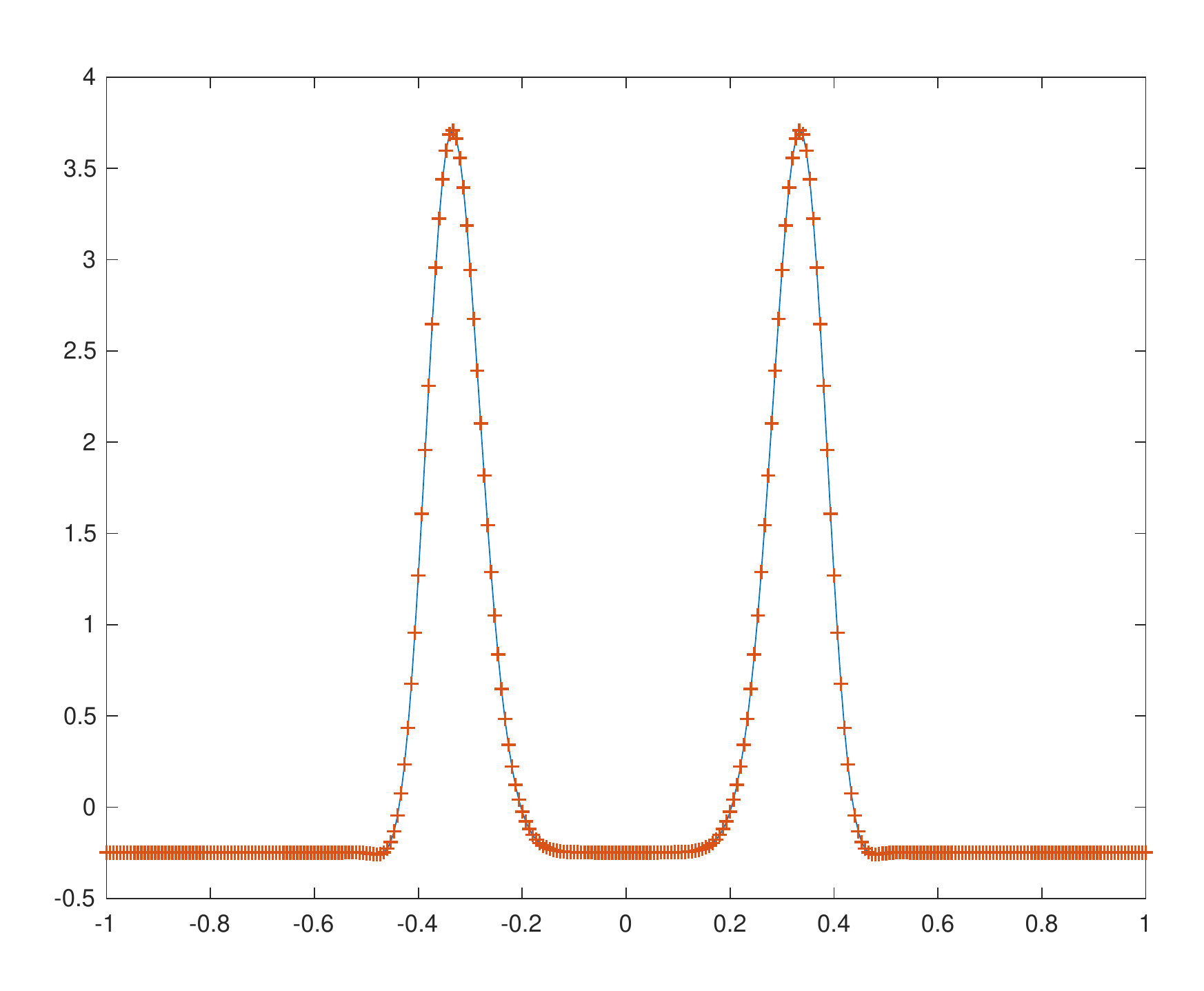} &
\includegraphics[scale=0.35]{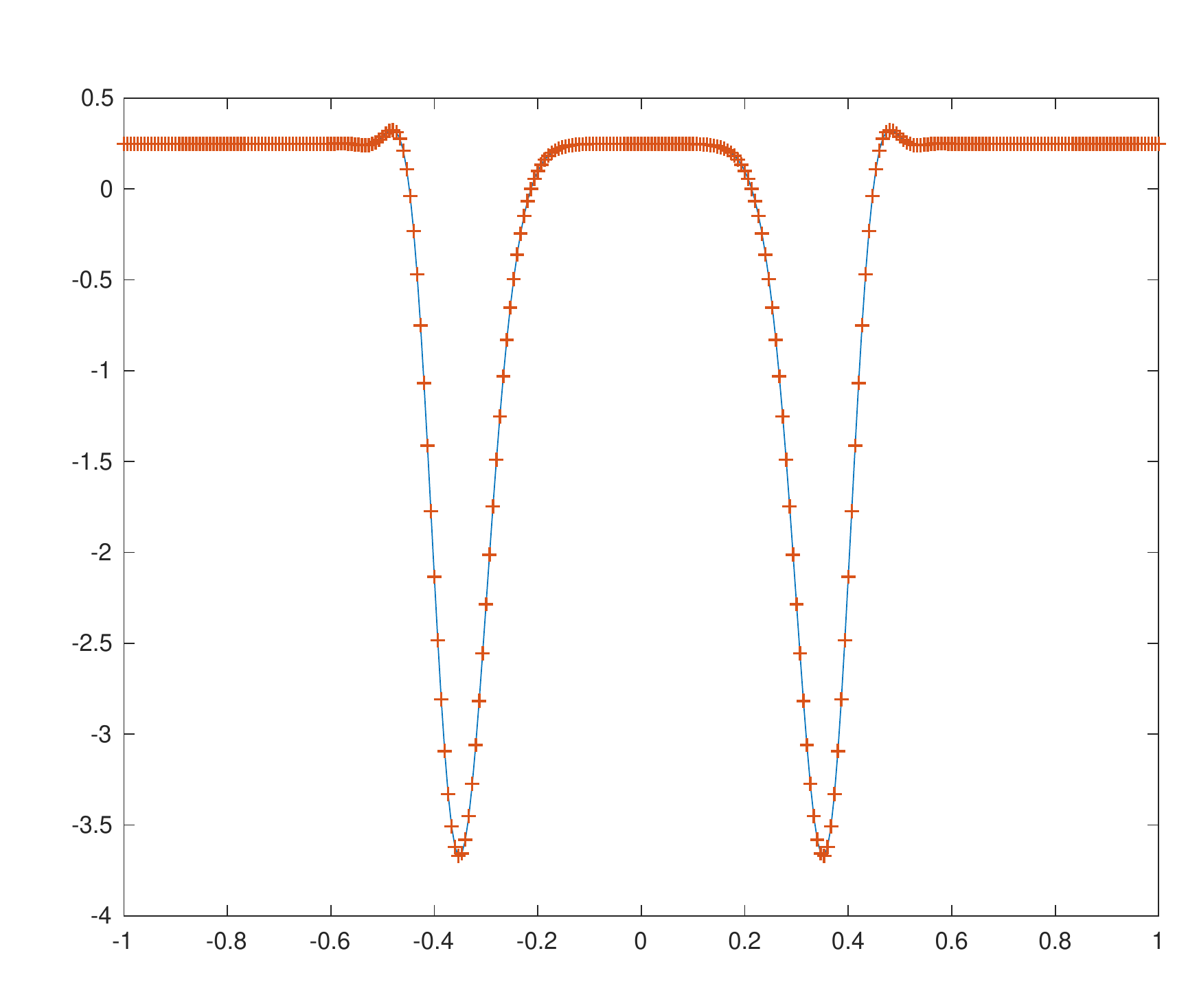} \\
\includegraphics[scale=0.35]{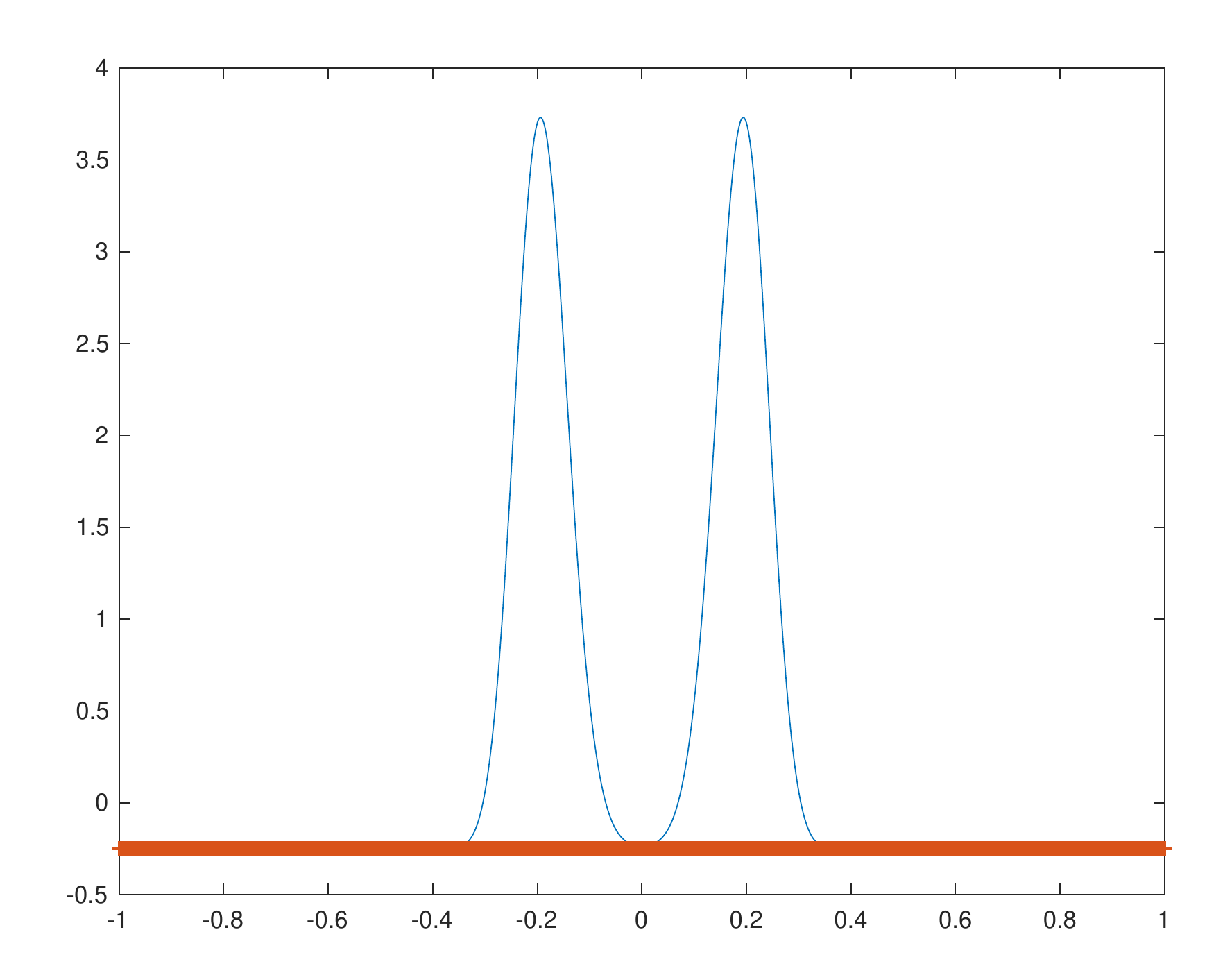} &
\includegraphics[scale=0.35]{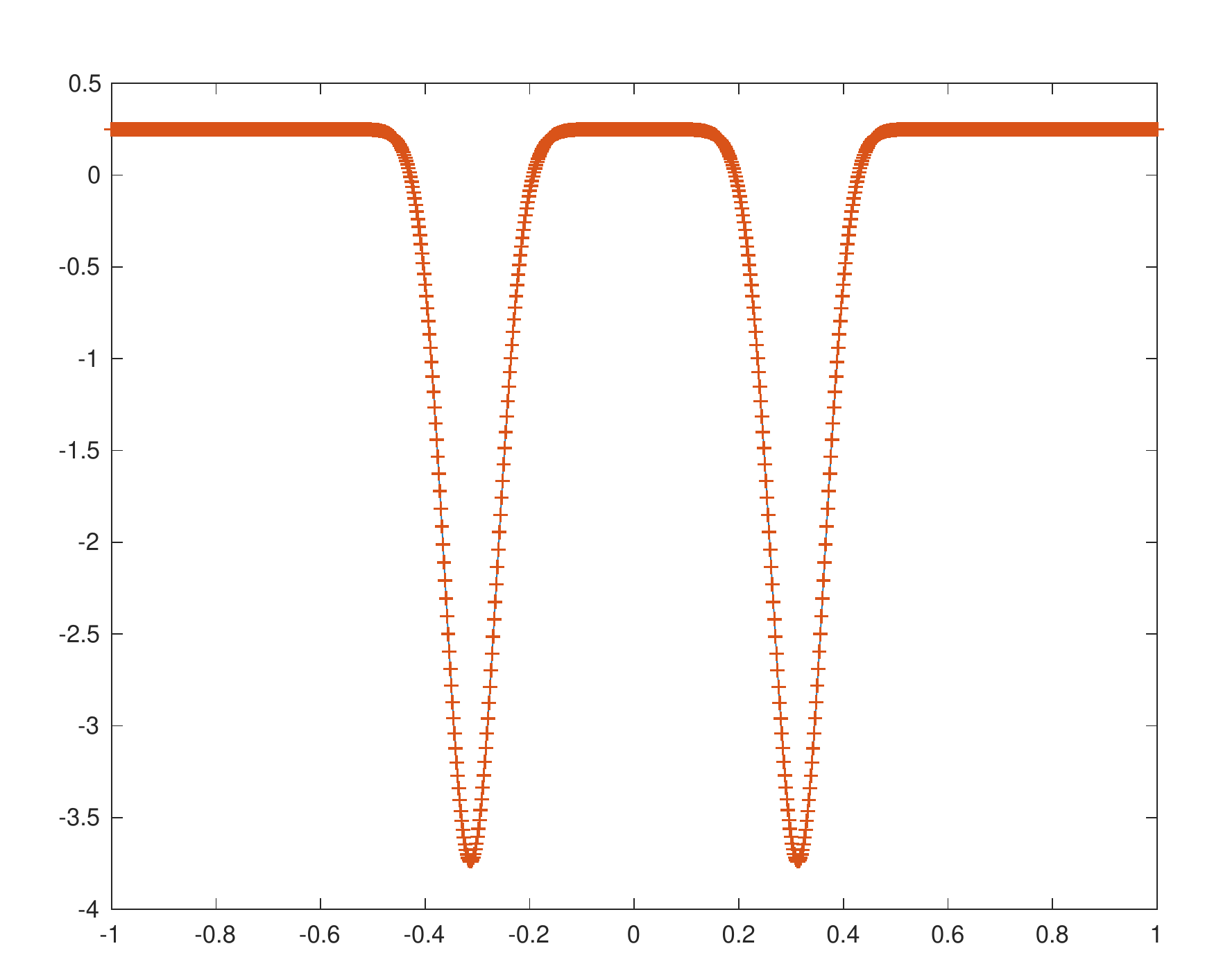}
\end{tabular}
\caption{Top: 
The Toeplitz wave at $t = 3.7$ (left) and Hankel wave at $t = 5.7$ (right).
The solid line is computed via \eqref{eq:Twave:T} (Toeplitz wave) and  \eqref{eq:Hwave:H} (Hankel wave),  and the crosses are computed via \eqref{eq:eq15} (Bessel expansion for Toeplitz wave) and  \eqref{eq:eq18} (Bessel expansion of Hankel wave).
 Top panel is computed with $R = 3$ and $N=301$.
 Bottom left: Toeplitz wave with
 $R=3$, $N=1001$, and $t= 2 \pi R + 1$.
 Bottom Right: The same as top right, except that $N=1001$.}
\label{fig:wave2}
\end{figure}

\section{Discussion and Future Work}
\label{sec:discussion}
Based on our previous analysis, we can make a number of observations.
\begin{itemize}
\item The term $X = \frac{1}{2(N+1)} (\alpha \bm{e} + \beta \cos(2j)\,\bm{w})$ is a function of the spatial  discretisation error in the sense that as $N$ increases, this term converges linearly to some limit that depends on the initial condition.
\item For the initial condition 
$$\bm{u}_0 = \frac{1}{\sqrt{2\pi} \sigma} e^{-\frac{x^2}{2\sigma^2}}, \quad x \in [-1, 1], \quad \sigma = 0.05$$
then we can show $$ \alpha = \bm{e}^\top\,\bm{u}_0 = \frac{1}{2}(N-1)$$
and $\beta = \bm{w}^\top \bm{u}_0$ is essentially 0 for even modest values of $N$, e.g. $N=101$. Thus $X$ is $\frac{1}{4}(1-\frac{2}{N+1})$ which converges to $\frac{1}{4}$ as $N \rightarrow \infty$. It is interesting that this term is not zero.
\item Note that the sum of the Toeplitz wave and the Hankel wave is the full wave and the $X$ term cancels. The constant component of the Toeplitz and Hankel waves as $N$ becomes large is $-\frac{1}{4}$ or $\frac{1}{4}$, respectively.
\item %\colorbox{yellow}{\parbox{0.98\textwidth}{
The user of our formulas  in \eqref{eq:eq15} and \eqref{eq:eq18}  can choose the value of the parameter $R \in \{ 0,1,2, \ldots, \}$.
This is a way to control the number of traversals of the domain of the Toeplitz and Hankel waves.%}}
\item %\colorbox{yellow}{\parbox{0.98\textwidth}{
For $t>2 \pi R$, the exact formulas that we present in \eqref{eq:eq15} and \eqref{eq:eq18}  are constant.
For $t< 2 \pi R$, the exact formulas  in \eqref{eq:eq15} and \eqref{eq:eq18} that we derive are exactly the Toeplitz and Hankel waves, and these sum exactly to the solution of the semi-discrete problem.%}}
\item  %\colorbox{yellow}{\parbox{0.98\textwidth}{
For the Toeplitz wave, the matrices $E_k$ and $E_k^\top, \> k=1,\cdots,N-1$ act as shift matrices that allow the wave to move. 
In the case of the Hankel wave, the corresponding matrices are the Hankel matrices $F_1, \cdots, F_{2k-1}$. 
%These matrices are Hankel and correspond to the $E$ matrices but in the backward diagonal sense. 
Thus one way to think about wave propagation is through these elementary Toeplitz and Hankel matrices.
%}}
\item There is a rich theory on the sums of Bessel functions of the form
\begin{equation}
f(z) = a_0 J_0(z) + 2 \sum_{k=1}^\infty a_k J_k(z).
\label{eq:eq20}
\end{equation}
These are called Neumann series \cite{Neumann1867}. Here 
$$ a_k = \frac{1}{2\pi i} \int_{|z|=c} f(t) \theta_k(t)\,dt,$$
$c$ is an appropriate contour, and the $\theta_k(t)$ are Neumann polynomials given by
\begin{eqnarray*}
\theta_0(t) &=& \frac{1}{t} \\
\theta_k(t) &=& \frac{1}{4} \sum_{l=0}^{[\frac{k}{2}]} \frac{(k-l-1)!}{l!} k \left(\frac{2}{t}\right)^{k-2l+1} .
\end{eqnarray*}

\end{itemize}

There are generalisations of (\ref{eq:eq20}) (see Watson \cite{Watson1922}) of the form
$$ z^\eta f(z) = a_0 J_\eta(z) + 2 \sum_{k=1}^\infty a_k J_{\eta +k}(z).$$
More recently Pogany and S\"{u}li \cite{Pogany09} have given integral representations for certain Neumann series including
$$\sum_{k=0}^\infty a_{k \eta} J_{\eta + 2k + 1}(z), \quad \eta \geq -\frac{1}{2}$$
$$a_{k \eta} = 2(\eta + 2k + 1) \int_0^\infty t^{-1} f(t) J_{\eta+2k+1}(t) dt.$$
Al-Jarrah et al. \cite{AlJarrah2002} have considered Neumann series in the context of the Fourier cosine transform
\begin{equation}
F_c(x) = \sqrt{\frac{2}{\pi}} \int_0^\infty f(t) \cos(xt) dt
\label{eq:eq21}
\end{equation}
and derived formulas for 
$$F_c(0) + 2 \sum_{k=1}^\infty F_c(2\pi k \eta)$$
for a given $f(t)$. We note that Bessel functions can be cast in the form of (\ref{eq:eq21}).

Since the Toeplitz and Hankel waves, defined over a finite number of traversals, can be considered as certain finite sums of even Bessel functions of the first kind, these are then finite Neumann series and some of this theory could be of interest in this context.
We speculate that generalisations of these ideas might be applicable to other applications of Bessel function expansions to waves \cite{movchan2009wave}, or to the analysis of  time series data arising in wave phenomena   \cite{pethiyagoda2017spectrograms}.

%\section{Discussion of variable coefficients}

\noindent  %\colorbox{yellow}{\parbox{0.98\textwidth}{
It is worth making some final remarks about the subject of inverse problems.
As mentioned in the Introduction, it is not the purpose of our article to study inverse problems. 
That would require extending our analysis here to a wave equation with variable coefficients, and also extending to higher dimensions.
Nonetheless, if we view our analysis here as a first step towards inverse problems, then we can make the following observations about how such an extension, in one dimension, may be carried out in future work.
Commonly, a wave propagates at different speeds, $c^2(x)$, at different spatial locations.
In an inverse problem, this speed $c^2(x)$ might initially be unknown, and the goal is to estimate $c^2(x)$ at many points $x$, based on known observational data about the wave.
One tool in that process of estimation, is called a forward solver, which given $c^2(x)$ at all locations, is then able to simulate the waves.
%}}

%Other tools are also required, such as a method of optimization that, together with the forward solver, can optimize the estimates of $c^2(x)$ so that the forward simulation fits the data as closely as possible.
%In this process of optimization, it might be advantageous to be able to control the number of reflections that are simulated.
%This is important for applications to inverse problems such as  imaging.
As a simple example, suppose that the domain consists of two halves, and that the wave speed is $c^2=1$ in the first half of the domain and $c^2=2$ in the second half. 
One possible model of this situation can come from scaling our finite matrix $\bm{K}$ by a diagonal matrix that encodes the wave speed. 
For example, $\bm{C} =
\textrm{diag}(1,\cdots,1,2,\cdots,2)$ and $\bm{L} = \bm{CK}$.
%\begin{equation}
%\bm{L} = \bm{C}\bm{K} =
%\left(
%\begin{tabular}{ccccc}
%$$ &  \\
%  & $\ddots$ &  \\
%  &  & $c^2(x)$ &  \\
% & &  & $\ddots$ &  \\
%  &  & &    & 
%\end{tabular}
%\right)
%\bm{K}.
%\end{equation}
%In our simple example
%\begin{equation}
%\bm{L} =
%\left(
%\begin{tabular}{ccccc}
%$1$ &  \\
%  & $1$ &  \\
%  &  & $\ddots$ &  \\
% & &  & $2$ &  \\
%  &  & &    & 2
%\end{tabular}
%\right)
%\left(
%\begin{tabular}{ccccc}
%$2$ & -1 \\
% -1 & $2$ & -1 \\
%  & $\ddots$ & $\ddots$ & $\ddots$ \\
% & & -1 & $2$ & -1 \\
%  &  & &   -1 & 2
%\end{tabular}
%\right)
%\label{eq:L}
%\end{equation}
It is tempting to apply our `Toeplitz-plus-Hankel' analysis to this situation.
However, the matrix $\bm{L}$ 
% in \eqref{eq:L} 
\textit{cannot} be written exactly as the sum of a Toeplitz matrix and a Hankel matrix.
Thus, it is impossible to write all matrix functions exactly as a sum of a Toeplitz part and a Hankel part.
We can still compute the solution to this  wave equation with variable speed as a matrix function by, for example, employing the wave-kernel matrix functions of Nadukandi \& Higham \cite{WaveKernelsHigham}, but we should not expect that matrix function to be simply Toeplitz-plus-Hankel.
In other words, we cannot expect to repeat our previous analysis for the constant coefficient case verbatim in this new variable coefficient case.

Nevertheless, it seems likely that  the analysis that we have presented  can be generalised to handle this setting of variable coefficients.
The starting point will be to again examine the algebraic structure of the solutions, and in analogy with the way we investigated $\bm{v} \bm{v}^{\top}$ here, particularly to examine the rank one projection matrices  coming from the spectral decomposition of the operator $\bm{L}$.
Encouragingly, the non-symmetric matrix in $\bm{L}$ is an example of a tridiagonal $k$--Toeplitz matrix (in our simple example, it is a $2$--Toeplitz matrix), for which exact expressions for the eigenvalues and for the eigenvectors are known \cite{alvarez2005some}.
The first step is to compute expressions for the rank one projection matrices coming from those eigenvectors, and this will be the subject of future work.

\section*{Acknowledgments}
This work was completed at the University of Oxford. We would like to thank Professors Endre S\"uli and Nick Trefethen of the Mathematical Institute  at the University of Oxford for their comments. The second author would like to thank Professor Nick Trefethen for hosting her visit at the Mathematical Institute in November and December 2018.

%\section*{References}
 \bibliographystyle{plain}
\bibliography{shevRefs}

\end{document}